\documentclass{amsart}

\usepackage[utf8]{inputenc}
\usepackage{amsmath, amssymb,amsthm,tikz}
\usepackage{tikz-cd}
\usepackage{tikz}
\usepackage{todonotes}
\usepackage{mathrsfs}
\usepackage{extarrows}
\usepackage{graphicx}
\usepackage{thmtools}
\usepackage{mathtools}
\usepackage{hyperref}

\usepackage[english]{babel}

\newcommand{\Z}{\mathbb{Z}}
\newcommand{\R}{\mathbb{R}}

\newcommand{\Q}{\mathbb{Q}}

\newcommand{\A}{\mathbb{A}}
\newcommand{\mcA}{\mathcal{A}}

\newcommand{\Oo}{\mathcal{O}}

\newcommand{\End}{\operatorname{End}}

\newcommand{\ab}{\operatorname{ab}}

\newcommand{\der}{\operatorname{der}}
\renewcommand{\sc}{\operatorname{sc}}

\newcommand{\rank}{\operatorname{rank}}

\newcommand{\Aut}{\operatorname{Aut}}

\newcommand{\Lie}{\operatorname{Lie}}

\newcommand{\Spec}{\operatorname{Spec}}

\newcommand{\Hom}{\operatorname{Hom}}
\newcommand{\Isom}{\operatorname{Isom}}

\newcommand{\Ann}{\operatorname{Ann}}

\newcommand{\Ext}{\mathbb{E}\operatorname{xt}}

\newcommand{\Sh}{\operatorname{Sh}}

\newcommand{\fppf}{\operatorname{fppf}}

\newcommand{\Gal}{\operatorname{Gal}}

\newcommand{\GSp}{\operatorname{GSp}}

\hypersetup{
    colorlinks,
    citecolor=blue,
    filecolor=blue,
    linkcolor=blue,
    urlcolor=blue
}

\renewcommand{\tilde}{\widetilde}
\renewcommand{\int}{\operatorname{int}}
\renewcommand{\mod}{\operatorname{mod}}

\newenvironment{subproof}[1][\proofname]{%
  \begin{proof}[#1]%
}{%
  \end{proof}%
}

\title{Finiteness of fppf cohomology
}
\author{Yujie Xu}
\address{Department of Mathematics, Harvard University}
\curraddr{}
\email{yujiex@math.harvard.edu}
\date{}

\numberwithin{equation}{section}
\newtheorem{theorem}[equation]{Theorem}
\newtheorem{prop}[equation]{Proposition}

\newtheorem{lem}[equation]{Lemma}
\newtheorem{Coro}[equation]{Corollary}

\theoremstyle{definition}
\newtheorem{Defn}[equation]{Definition}
\newtheorem{remark}[equation]{Remark}
\newtheorem{numberedparagraph}[equation]{}

\begin{document}

\maketitle

\begin{abstract}
Let $R$ be a Henselian DVR with finite residue field. 
Let $G$ be a 
finite type, flat $R$-group scheme (not necessarily commutative) with smooth generic fiber. We show that $H^1_{\fppf}(\Spec R,G)$ is finite. We then give an application of the global analogue of this finiteness result to PEL type integral models of Shimura varieties. 
\end{abstract}

\tableofcontents

\section{Introduction}
We prove the finiteness of $H^1$ for the fppf cohomology in greater generality than the known results in existing literature. 
\begin{theorem}\label{intro-global-thm}(Theorem \ref{finiteness-OKS-flat-gp})
Let $K$ be a number field, $S$ a finite set of places of $K$, and $G$ a flat group scheme of finite type over $\Oo_{K,S}$ such that the connected component of identity in the generic fibre of $G$ is reductive. Then the set $H^1(\Spec\Oo_{K,S},G)$ is finite. 
\end{theorem}
We also give an application of Theorem \ref{intro-global-thm} to the Hodge morphism 
\[\mathscr{S}_K(G,X)\to\mathscr{S}_{K'}(\GSp,S^{\pm})\]
where $\mathscr{S}_K(G,X)$ is a PEL type integral model for the Shimura variety $\Sh_K(G,X)$ and $\mathscr{S}_{K'}(\GSp,S^{\pm})$ is a Siegel integral model (see section \ref{PEL-application-section}). 
The proof of the above theorem \ref{intro-global-thm} reduces to proving the following local analogue.
\begin{theorem}\label{intro-local-thm}(Corollary \ref{local-fppf-finiteness})
Let $R$ be a Henselian DVR with finite residue field. For $n$ large enough, the natural morphism,
\[\delta_n\colon H^1_{\fppf}(\Spec R,G)\to H^1_{\fppf}(\Spec R/\pi^nR,G)\]
is injective. Consequently, $H^1_{\fppf}(\Spec R,G)$ is finite.
\end{theorem}
The proof of Theorem \ref{intro-local-thm} amounts to applying results of \cite{Gabber-Ramero} on the vanishing of $\Ext^1$ of certain cotangent complexes, which measures obstruction to deformation of torsors. In fact, we prove a slightly stronger result than Theorem \ref{intro-local-thm} (see Proposition \ref{Henselian-pair-injective-lemma}).

\section{The local case}
\begin{numberedparagraph}
Let $(R,\pi)$ be a Henselian pair. Let $G$ be a finite type, flat $R$-group scheme (not necessarily commutative) that is smooth over $R[1/\pi]$.

We place ourselves in the situation of \cite[5.4.7.]{Gabber-Ramero}, where 
$F=R[X_1,\cdots,X_N]$ is a free $R$-algebra of finite type, and $J\subset F$ is a finitely generated ideal such that $S=F/J$ and $G\cong \Spec S$. Let $P,P'$ be two arbitrary $G$-torsors and $\Isom(P,P')$ be the Isom scheme. 
To differentiate the notations, we will write $G\cong \Spec F_G/J_G$ and 
$\Isom(P,P')\cong \Spec F_{\Isom(P,P')}/J_{\Isom(P,P')}$. For any $a:=(a_1,\cdots, a_N)\in R^N$, let $\mathfrak{p}_a\subset F$ be the ideal generated by $(X_1-a_1,\cdots,X_N-a_N)$. Let $I\subset R$ be an ideal. We set the following ideal $H_R(F,J)$ of $F$ as in \cite[Definition 5.4.1]{Gabber-Ramero}, 
\begin{equation}\label{Definition-HRFJ}
H_R(F,J):=\Ann_F\Ext_S^1(\mathbb{L}_{S/R},J/J^2),
\end{equation}
where $\mathbb{L}_{S/R}$ is the cotangent complex. 
\begin{lem}\cite[Lemma 5.4.8]{Gabber-Ramero}
Let $n,h\in\mathbb{N}$ and $a\in R^N$ such that 
\begin{equation}\label{Gabber-Ramero-5.4.9}
    t^h\in H_R(F,J)+\mathfrak{p}_a\qquad J\subset\mathfrak{p}_a+t^nIF\qquad n>2h.
\end{equation}
Then there exists $b\in R^N$ such that 
\[b-a\in t^{n-h}IR^N\qquad\text{and}\qquad J\subset \mathfrak{p}_b+(t^{n-h}I)^2F.\]
\end{lem}
\begin{Defn}
Let $h_{G}$ denote the smallest integer satisfying the condition of \ref{Gabber-Ramero-5.4.9} 
where $F=F_G$ and $J=J_G$. 
\end{Defn}
One can likewise define the integer $h_{\Isom(P,P')}$.

\begin{prop}\label{Henselian-pair-injective-lemma}
Let $(R,\pi)$ be a general Henselian pair with $\pi$ a non-zero-divisor. Let $G$ be a finite type, flat $R$-group scheme (not necessarily commutative) that is smooth over $R[1/\pi]$. 
There exists an $N$, which only depends on $G$
, such that 
$H^1(R, G) \to H^1(R/{\pi}^N, G)$ is injective.  
\end{prop}
\begin{proof}
Suppose $[P],[P']\in H^1(\Spec R,G)$ map to the same class $[P_n]\in H^1(\Spec R/\pi^nR,G)$ under $\delta_n$. Thus we have $[\underline{\Isom}(P,P')]\in\delta_n^{-1}([\underline{\Isom}(P_n,P_n)])$. 
Now $\underline{\Isom}(P_n,P_n)$ is the trivial 
$\Aut_{G_n}(P_n)$-torsor 
on $\Spec R/\pi^n$ 
, i.e. we have a section
\[\sigma_n\colon\Spec R/\pi^n\to \underline{\Isom}(P_n,P_n).\] 
Thus $\underline{\Isom}(P,P')(R/\pi^n)\neq\varnothing$. By \cite[Lemma 5.4.13]{Gabber-Ramero} whose notations we also adopt, as long as $n>2h_{\underline{\Isom}(P,P'
)}$, there exists a lifting 
\[b_{P,P'}=\Tilde{\sigma_{P,P'}}\colon \Spec R\to \underline{\Isom}(P,P'),\] 
such that 
$\Tilde{\sigma_{P,P'}}|_{\Spec R/\pi^{n-h_{\underline{\Isom}(P,P')}}}=\sigma_{P,P'}|_{\Spec R/\pi^{n-h_{\underline{\Isom}(P,P')}}}$. In other words, as long as $n>2h_{\underline{\Isom}(P,P
')}$, we have $\underline{\Isom}(P,P')(R)\neq\varnothing$, which implies that $\underline{\Isom}(P,P')$ is a trivial $\Aut_G(P)
$-torsor on $\Spec R$, thus $P\cong P'$ as $G$-torsors on $\Spec R$, and thus $[P]=[P']\in H^1(\Spec R,G)$. Thus the map $\delta_n$ is injective. 

Thus it only remains to show that one can indeed find such a uniform 
integer $n$ that works for all $G$-torsors $P,P'$ on $\Spec R$, i.e. we want to find an 
integer $n$ such that 
$n>2h_{\underline{\Isom}(P,P')}$ 
for all $G$-torsors $P,P'$ on $\Spec R$. By the following Lemma \ref{affine-main-lemma}, this is indeed the case; therefore we have our desired result.
\end{proof}

\begin{Coro}\label{local-fppf-finiteness}
Let $R$ be a Henselian DVR with finite residue field. Let $G$ be a 
finite type, flat $R$-group scheme (not necessarily commutative) with smooth generic fiber\footnote{Technically speaking, we also need to assume that $G$ is affine, but 
the result holds in the stated generality by essentially the same proof, see \ref{quasi-affine-case} and 
\ref{algebraic-spaces-cotangent-complexes-remark}.}.
For $n$ large enough, the natural morphism,
\[\delta_n\colon H^1_{\fppf}(\Spec R,G)\to H^1_{\fppf}(\Spec R/\pi^nR,G)\]
is injective. Consequently, $H^1_{\fppf}(\Spec R,G)$ is finite.
\end{Coro}
\begin{proof}
This follows directly from Proposition \ref{Henselian-pair-injective-lemma}. 
\end{proof}
\end{numberedparagraph}

\begin{numberedparagraph}
Recall the ideal $H_R(F,J)$ from 
\ref{Definition-HRFJ}
\begin{equation}\label{defining-HRFJ}
H_R(F,J):=\Ann_F\Ext^1_S(\mathbb{L}_{S/R},J/J^2)
\end{equation}
By \cite[Lemma 5.4.2(iii)]{Gabber-Ramero}, we have
\[H_R(F,J)=\bigcap\limits_{\text{all $S$-modules $N$}}\Ann_F \Ext_S^1(\mathbb{L}_{S/R},N)\]
Therefore, we can rewrite $H_R(F,J):=H_R(F/J):=H_R(S)$, i.e. the ideal $H_R(F,J)$ only depends on $S=F/J$, and is therefore independent of the presentation of $S$. In order to prove the desired Lemma \ref{affine-main-lemma}, our goal is to show that
\begin{equation}
    H_R(F_G/J_G)=H_R(F_{\mathrm{Isom}}/J_{\mathrm{Isom}})
\end{equation}
To this end, we prove the following series of lemmas.
\end{numberedparagraph}

\begin{lem}\label{HR-flat-base-change}
The formation of $H_R(F,J)$ commutes with flat base change. i.e. \[H_R(S)\otimes_RR'\cong H_{R'}(S\otimes_RR')\]
\end{lem}
\begin{proof}
First note that the formation of the cotangent complexes and the formation of $\Ext_S^1$ both commute with flat base change, i.e. we have
\[\Ext_S^1(\mathbb{L}_{S/R},J/J^2)\otimes_RR'\cong \Ext_{S\otimes_R R'}^1(\mathbb{L}_{S\otimes_RR'},J'/{J'}^2)\]
where $J'=J\otimes_RR'$. 

Since $G$ and $\Isom$ are complete local intersections, their cotangent complexes $\mathbb{L}_{S/R}$ are perfect complexes concentrated in degrees $0$ and $1$. Therefore $\Ext_S^1(\mathbb{L}_{S/R},J/J^2)$ is a finite-type $S$-module. Therefore, the formation of the annihilator $\Ann\Ext_S^1(\mathbb{L}_{S/R},J/J^2)$ commutes with flat base change. Thus we have
\[\big(\Ann_{F}\Ext_S^1(\mathbb{L}_{S/R},J/J^2)\big)\otimes_RR'\cong \Ann_{F'}\Ext_{S\otimes_R R'}^1(\mathbb{L}_{S\otimes_RR'},J'/{J'}^2)\]
where $F'=F\otimes_RR'$. Since the ideal $H_R(F,J)$ is independent of the presentation in terms of $S=F/J$, we have
\[H_R(S)\otimes_RR'\cong H_{R'}(S\otimes_RR')\]
i.e. the formation of the ideal $H_R(F,J)$ commutes with flat base change.
\end{proof}

\begin{lem}\label{affine-main-lemma}
For any $G$-torsors $P,P'$ on $\Spec R$,
\[h_{\underline{\Isom}(P,P')}=h_{\Aut_G(P)}=h_P=h_{P'}=h_G\]
\end{lem}
\begin{proof}
It suffices to show that $h_{\underline{\Isom}(P,P')}=h_G$, the other equalities follow by exactly the same argument. To check whether $N\geq h$, by \cite[Lemma 5.4.8]{Gabber-Ramero}, we want to check whether $S/H_R(S)$ is killed by $\pi^N$. We have the following diagram
\begin{equation}\label{diagram-for-G}
\begin{tikzcd}S/H_R(S)\arrow[]{r}{\pi^N}\arrow[hook]{d}{}&S/H_R(S)\arrow[hook]{d}{}\\
\big(S/H_R(S)\big)\otimes_{R}R'\arrow[]{r}{\pi^N}&\big(S/H_R(S)\big)\otimes_{R}R'
\end{tikzcd}
\end{equation}
The map $\pi^N$ in the second row being zero implies that in the first row being zero. By Lemma \ref{HR-flat-base-change}, we have
\[\Big(S_G/H_R(S_G)\Big)\otimes_RR'\cong S_{G\otimes_RR'}/H_{R'}(S_{G\otimes_RR'})\]
In particular, if $N>0$ such that $\pi^N$ kills $S_G/H_R(S_G)$, then $\pi^N$ kills 
\[ S_{G\otimes_RR'}/H_{R'}(S_{G\otimes_RR'})\]
by the above diagram. Since $G\otimes_RR'\cong\Isom(P,P')\otimes_RR'$, we have that $\pi^N$ also kills 
\[S_{\Isom(P,P')\otimes_RR'}/H_R(S_{\Isom(P,P')\otimes_RR'})\cong S_{G\otimes_RR'}/H_R(S_{G\otimes_RR'})\]
By the analogous diagram for $\Isom(P,P')$ as \ref{diagram-for-G}, $\pi^N$ also kills $S_{\Isom}/H_R(S_{\Isom})$. Thus $N\geq h_{\Isom}$. We simply take the minimal such $N$ to obtain $h_{G}=h_{\Isom}$. Likewise, the same $h$ holds for all inner forms of $G$. 
\end{proof}
Along the way we have also proven the following:
\begin{Coro}
$H_R(F_G/J_G)=H_R(F_{\mathrm{Isom}}/J_{\mathrm{Isom}})$. 
\end{Coro}

\begin{numberedparagraph}\label{quasi-affine-case}
In the case where $G$ is quasi-affine instead of affine, we can proceed similarly. We write
\[G=\bigcup\limits_{i=1}^m(\Spec S_i)\]
Let $\Oo_G$ be the structure sheaf of $G$. Let $\mathcal{H}_R(G)$ be the ideal sheaf given by $H_R(S_i)$ on each affine open $\Spec S_i$. Using the same argument as in Lemma \ref{HR-flat-base-change} on affines, one can see that the formation of the ideal sheaf $\mathcal{H}_R(-)$ commutes with flat base change as well. 
Consider the relative Spec
\[\underline{\Spec}_{\Oo_G}(\Oo_G/\mathcal{H}_R(G))\]
which is a closed subscheme of $G$. In particular, we also have
\begin{lem}Let $G,G'$ be fppf locally isomorphic. Then:\\ 
$\pi^N$ is zero on $\underline{\Spec}_{\Oo_G}(\Oo_G/\mathcal{H}_R(G))\Longleftrightarrow\pi^N$ is zero on $\underline{\Spec}_{\Oo_{G'}}(\Oo_{G'}/\mathcal{H}_R(G'))$
\end{lem}
\begin{proof}
One can 
work Zariski locally to reduce to the affine case, which is proven in Lemma \ref{affine-main-lemma}. 
\end{proof}
\begin{remark}
When $G$ is quasi-affine, $\Isom(P,P')$ is representable by a quasi-affine scheme as well. This isn't necessarily true for groups $G$ that are not quasi-affine.
\end{remark}
\end{numberedparagraph}

\begin{remark}\label{algebraic-spaces-cotangent-complexes-remark}
We can relax the conditions on $G$ even further, and let $G$ be any flat, finite type, generically smooth group scheme. Now $\Isom(P,P')$ is an algebraic space. Cotangent complexes for algebraic spaces are defined in \cite{Laumon-Moret-Bailly}. One can then check, via a similar argument as \ref{HR-flat-base-change}, that the formation of $\mathcal{H}_R(-)$ in the category of algebraic spaces commutes with \'etale localization.
\end{remark}

\section{The global case}
\begin{prop}\label{finite-fibre-OKS-to-K}
Let $K$ be a number field, $S$ a finite set of places of $K$, and $G$ a connected reductive group over $\Oo_{K,S}$. Then the fibres of the map
\[H^1(\Spec\Oo_{K,S},G)\to H^1(K,G)\]
are finite.
\end{prop}
\begin{proof}
Suppose $P,P'$ are $G$-torsors over $\Oo_{K,S}$ such that $P|_K\xrightarrow{\sim} P'|_K$. 
For any prime $v\notin S$ of $K$, we denote $P_v:=P|_{\Oo_{K_v}}$ and $P'_v:=P'|_{\Oo_{K_v}}$. 
By Lang's Lemma, such $P_v$ and $P'_v$ are trivial, and in particular they're isomorphic. 
The isomorphism $P|_K\xrightarrow{\sim}P'|_{K}$ extends over $\Spec\Oo_{K,T}$ for some finite set of primes $T\supset S$.\\
For $v\notin S$ there are tautological ismorphisms $P_v|_{K_v}\xrightarrow{\sim} P|_{K_v}$, 
and $P'$ can be constructed by modifying these isomorphisms by an element of $g_v\in G(K_v)$ for $v\in T-S$ and gluing the $P_v$ to $P|_{\Oo_{K,T}}$ along the resulting isomorphisms. If $g_v\in G(\Oo_{K_v})$ for $v\in T-S$, or all the $g_v$ are equal to a single element $g\in G(\Oo_{K,T})$, then $P$ and $P'$ are isomorphic. 

Thus one sees that there is a surjection from the set 
\begin{equation}\label{9.1}
    \lim\limits_{T}G(\Oo_{K,S})\backslash \prod\limits_{v\in T-S}G(K_v)/G(\Oo_{K_v})=G(\Oo_{K,S})\backslash G(\A_f^S)/\prod\limits_{v\notin S}G(K_v)
\end{equation}
to any fibre of the map $H^1(\Spec\Oo_{K,S},G)\to H^1(K,G)$. Here $\A_f^S$ denotes the adeles of $K$ with trivial components at primes in $S$. (In fact it is not hard to see that this map is a bijection.)

By a result of Borel and Harish-Chandra \cite{Borel-Harish-Chandra}, the set in \ref{9.1} is finite, and thus the map $H^1(\Spec\Oo_{K,S},G)\to H^1(K,G)$ has finite fibres.
\end{proof}

Now we can combine all the above results and prove the following. 
\begin{prop}\label{finiteness-global-OKS}
Let $K$ be a number field, $S$ a finite set of places of $K$, and $G$ a smooth group over $\Oo_{K,S}$ such that the connected component of the identity $G^0\subset G$ is reductive, and $G/G^0$ is \'etale. Then $H^1(\Spec\Oo_{K,S},G)$ is finite.
\end{prop}
\begin{proof}
If $G$ is finite \'etale, by the global Poitou-Tate duality, we have $H^i(\Spec\Oo_{K,S},G)$ is finite for $i=0,1,2$. 
Thus by d\'evissage 
it suffices to consider the case when $G$ is connected reductive, which we now assume. 
Note that by Lemma \ref{finite-fibre-OKS-to-K}, we may replace $S$ by a larger finite set $S'$. 
Thus we may assume that the simply connected cover $G^{\sc}$ of the derived group $G^{\der}$ of $G$ is \'etale over $G^{\sc}$ and hence of order invertible on $\Spec\Oo_{K,S}$. We may further assume that the abelianization $G^{\ab}$ splits over a field which is unramified at $v\notin S$. 

We now consider a number of cases. 

(I) Suppose $G$ is a torus. Let $L/K$ be a finite extension over which $G$ splits. Our assumptions imply that we may take $L/K$ to be unramified at primes $v\notin S$. 
By Hilbert Theorem 90 and our assumption on $G^{\ab}$, we have 
\[H^1(\Spec\Oo_{K,S},G)=H^1(\Gal(L/K),G(\Oo_{L,S})).\]

As an abelian group $G(\Oo_{L,S})$ is a product of copies of $\Oo_{L,S}^{\times}$. In particular, this is a finitely generated abelian group. 
Hence $H^1(\Gal(L/K),G(\Oo_{L,S}))$ is finite; it is a finitely generated abelian group killed by the order of $\Gal(L/K)$.

(II) Suppose $G=G^{\sc}$. Then by Lemma \ref{finite-fibre-OKS-to-K}, it suffices to prove that $H^1(K,G)$ is finite, which follows from the Hasse principle.

This, together with the remark above when $G$ is finite, implies the lemma when $G$ is simply connected.

Finally, combining this with the case of a torus implies the Lemma in general. 
\end{proof}

\begin{theorem}\label{finiteness-OKS-flat-gp}
Let $K$ be a number field, $S$ a finite set of places of $K$, and $G$ a flat group scheme of finite type over $\Oo_{K,S}$ such that the connected component of identity in the generic fibre of $G$ is reductive. Then the set $H^1(\Spec\Oo_{K,S},G)$ is finite. 
\end{theorem}
\begin{proof}
Let $T\supset S$ 
be a finite set of places of $K$ such that $G|_{\Oo_{K,T}}$ satisfies the conditions of Proposition \ref{finiteness-global-OKS}. We consider the map
\begin{equation}\label{Prop-11-map}
    H^1(\Spec\Oo_{K,S},G)\to H^1(\Spec\Oo_{K,T},G)\times\prod\limits_{v\in T-S}H^1(\Oo_{K_v},G).
\end{equation}
By Proposition \ref{finiteness-global-OKS}, 
$H^1(\Spec\Oo_{K,T},G)$ is finite. And by Proposition \ref{local-fppf-finiteness}, each $H^1(\Oo_{K_v},G)$ is finite, and thus $\prod\limits_{v\in T-S}H^1(\Oo_{K_v},G)$ as a finite product of finite sets is clearly finite. Thus the RHS of the above map (\ref{Prop-11-map}) is finite. Thus in order to show that $H^1(\Spec\Oo_{K,S}, G)$ is finite, it suffices to show that the above map (\ref{Prop-11-map}) has finite fibres. This can be shown using exactly the same argument as in Lemma \ref{finite-fibre-OKS-to-K}--the only difference is that $P$ and $P'$ are isomorphic over $\Oo_{K_v}$ for $v\in T-S$, hence trivial. 
\end{proof}

\section{An application to integral models of Shimura varieties of PEL type}\label{PEL-application-section}

From now on, we fix a PEL datum, i.e. we fix a semisimple $\Q$-algebra $B$ with a maximal order $\Oo_B$. The algebra $B$ is endowed with a positive involution $*$ and $\Oo_B$ is stable under the involution $*$. 
Let $(\mcA,\lambda,\iota)$ be an abelian scheme with $\Oo_B$-action $\iota$ and a symmetric $\Oo_B$-linear polarization $\lambda$. 

Let $\underline{\Aut}(\mcA)$ be the $\Z$-group 
whose points in a $\Z$-algebra $R$ are given by 
$\underline{\Aut}(\mcA)(R)=(\End_S\mcA\otimes_{\Z}R)^{\times}$. 
Let $L$ be the $\Z$-scheme of linear maps $\Oo_B\to \End_S\mcA$, i.e. we have 
$L:=\underline{\Hom}_{\Z}(\Oo_B,\End_S\mcA)$, 
whose $R$-points (for any $\Z$-algebra $R$) 
are given by 
\[L(R):=\underline{\Hom}_{\Z\otimes R}(\Oo_B\otimes_{\Z}R,\End_S\mcA\otimes_{\Z}R)\cong \underline{\Hom}_{\Z}(\Oo_B,\End_S\mcA)\otimes_{\Z}R\]
Now, for a fixed endomorphism structure $\iota:\Oo_B\to \End_S(\mcA)$, 
we have $\iota\in L(\Z)$.

For any $b\in \underline{\Aut}(\mcA)(R)$, we can define an element $\iota_b\in L(R)$ given by 
\[\iota_b(a)=b\iota(a)b^{-1}\quad\mbox{ for any }a\in\Oo_B\otimes_{\Z}R.\]
(Note that this construction is inspired by explicit computations of matrices.)

This gives us a map of $\Z$-schemes $\psi: \underline{\Aut}(\mcA)\to L$, 
which on the level of $R$-points is given by 
\begin{align*}
\psi(R): \underline{\Aut}(\mcA)(R)&\to L(R)\\
b&\mapsto \iota_b
\end{align*}
Next we construct a subgroup scheme $H$ of $\underline{\Aut}(\mcA)$. 
We start with our fixed $\iota: \Oo_B\to \End_S(\mcA)$. 
Denote by $B'$ the commutant of $(\iota\otimes\Q)(B)$ in $\End_S\mcA\otimes\Q$. 
We also denote $\Oo_{B'}:=B'\cap \End_S\mcA$.  
We define an algebraic group $H$ over $\Z$, whose $R$-points (for $\Z$-algebra $R$) 
are given by
$H(R):=(\Oo_{B'}\otimes_{\Z}R)^{\times}=((B'\cap \End_S\mcA)\otimes_{\Z}R)^{\times}$. 
Note that the group $H$ is smooth, with connected geometric fibres. 
To see this, denote $\rank_{\Z}\Oo_{B'}=n$. Consider the affine space $\A^n$ over $\Z$ corresponding to $\Oo_{B'}$. We can embed $H\subset \A^n$ (this can be seen explicitly by writing down the equations of $H$).
Thus $H$ is open in the affine space $\A^n$. 
Thus $H$ is smooth with connected geometric fibres. 
The generic fibre $H_{\Q}$ is reductive, and so $H_{\Z[1/N]}$ is reductive for some integer $N$.

Note that the semisimple algebra $\End_S(\mcA)$ is equipped with the Rosatti involution $\dagger$, 
and $\Oo_{B'}$ inherits an involution from $\End_S(\mcA)$. 
For any involution $*$ of $\Oo_{B'}$, we denote by $\mathcal{O}_{B'}^*\subset\Oo_{B'}$ the subgroup fixed by $*$. 
Note that $\Oo_{B'}^*$ is a saturated subgroup of $\Oo_{B'}$. 
We define
$H^*:=H\cap\Oo_{B'}^*$, 
where the intersection is the scheme-theoretic intersection. 
In particular, when we take the specific involution $*=\dagger$ (the Rosati involution), we obtain the $\Z$-scheme 
$H^{\dagger}:=H\cap \Oo_{B'}^{\dagger}$. 
We also denote by $\Aut(\mcA)^{\dagger}$ the part of $\Aut(\mcA)$ fixed by $\dagger$. ($\Aut(\mcA)^{\dagger}$ is also viewed as a $\Z$-scheme.) 
Note that $H^{\dagger}$ is a sub-scheme of $\Aut(\mcA)^{\dagger}$. 
Since $H\subset\underline{\Aut}(\mcA)$ commutes with the image of $\iota$, this induces a map $\underline{\Aut}(\mcA)/H\to L$, identifying $\underline{\Aut}(\mcA)/H$ with a locally closed subscheme of $L$.

Let $I\subset\Aut(\mcA)$ be the algebraic group over $\Z$ consisting of elements $m$ with $mm^{\dagger}=1$. 
By positivity of the Rosatti involution, 
$I$ is compact, i.e. $I(\R)$ is compact. 
We define a map of $\Z$-schemes 
    $\wp: \Aut(\mcA)\to\Aut(\mcA)^{\dagger}$ given by $m\mapsto mm^{\dagger}$ 
and we denote by $\tilde{H}\subset \Aut(A)$ the preimage of $H^{\dagger}$ under this map.
By Definition, 
$I=\ker\wp\subset \Aut(\mcA)$. 
In fact, 
$I\subset \tilde{H}$.

Now we recall some facts about torsors. For a scheme $T$ and a group scheme $G$ over $T$, we denote by $H^1(G,T)$ the set of isomorphism classes of $G$-torsors on $T$. \\

\begin{lem}\label{Lemma 5}
Let $G$ be a flat group scheme of finite type over a scheme $T$ and $G'\subset G$ a closed, flat subgroup. Then there is a sequence of maps
\[G(T)\to G/G'(T)\xrightarrow{\delta}H^1(T,G')\]
where if $x,x'\in G/G'(T)$ then $\delta(x)=\delta(x')$ if and only if there exists a $g\in G(T)$ with $g\cdot x=x'$. 
\end{lem}
\begin{proof}
For $x\in G/G'(T)$ the element $\delta(x)$ is the class of the $G$-torsor given by $F_x$, the fibre over $x$ of the map $G\to G/G'$.\\
If $x,x'\in G/G'(T)$, we consider the subscheme of $G$ whose points in a $T$-scheme $S$ are given by the set of $y\in G(S)$ such that $y\cdot x=x'$. This subscheme is a $G$-torsor, isomorphic to the scheme of isomorphisms $\underline{\Isom}(F_x,F_x')$. In particular, $F_x$ and $F_x'$ are isomorphic if and only if this $G$-torsor is trivial; that is, if and  only if $x$ and $x'$ differ by a point of $G(T)$.
\end{proof}

Here we recall the following corollary of the Noether-Skolem Theorem
\begin{lem}\label{Noether-Skolem}
For any two $k$-algebra homomorphisms $f,g$ from a simple $k$-algebra $A$ to a central simple $k$-algebra $B$, there exists $b\in B^{\times}$ such that $f(x)=bg(x)b^{-1}$ for all $x\in A$.
\end{lem}

We will also need the following easy observation: 
For any $\iota:\Oo_B\hookrightarrow\End(\mcA)$ given in the PEL Shimura data, the image of $\iota$ lies in $\End(\mcA)^{\dagger}$, where $\dagger$ is the Rosati involution induced from the $\Oo_B$-linear polarization $\lambda$ in our Shimura datum. (i.e. the image of $\iota$ is stable under $\dagger$, by which we mean
\[\iota(a)^{\dagger}=\iota(a)\qquad\mbox{for all }a\in\Oo_B. )\]
Moreover, $\dagger$ induces an anti-involution on $\Oo_B$ via $\iota$.

\begin{theorem}\label{main_PEL_thm}
The morphism on the level of integral models 
\[\Phi:\mathscr{S}_K(G,X)\to\mathscr{S}_{K'}(\GSp,S^{\pm})\]
induced (on the level of points) by $(\mcA,\lambda,\iota)\mapsto (\mcA,\lambda)$ has finite fibres. 
\end{theorem}
\begin{proof}
We do it in steps. \\
\textit{Step I: The Theorem reduces to showing that $I(\Z)\backslash\tilde{H}(\Z)/H(\Z)$ is finite.}

\begin{subproof}
Suppose $(\mcA,\lambda,\iota)$ and $(\mcA,\lambda,\iota')$ are two triples (on the left-hand side of $\Phi$) mapping to the same principally polarized abelian scheme $(\mcA,\lambda)$ on the right-hand side. 
By the Noether-Skolem Theorem, $\iota'(a)=b\iota(a)b^{-1}$ for some $b\in \underline{\Aut}(\mcA)(\Q)$ and for all $a\in\Oo_B$.

Since the images of both $\iota'$ and $\iota$ are stable under $\dagger$, we have for any $a\in\Oo_B$,
\[\iota'(a)=\iota'(a)^{\dagger}=(b\iota(a)b^{-1})^{\dagger}=(b^{-1})^{\dagger}\iota(a)^{\dagger}b^{\dagger}=(b^{-1})^{\dagger}\iota(a)b^{\dagger}.\]
Combined with the original identity
$\iota'(a)=b\iota(a)b^{-1}$, we obtain for any $a\in\Oo_B$,
\[b\iota(a)b^{-1}=(b^{-1})^{\dagger}\iota(a)b^{\dagger}.\]
Since the polarization $\lambda$ is required to be $\Oo_L$-linear and symmetric,  $(b^{-1})^{\dagger}=(b^{\dagger})^{-1}$, and thus 
for any $a\in\Oo_B$,
\[(b^{\dagger}b)\iota(a)(b^{\dagger}b)^{-1}=\iota(a).\qquad(*)\]
Now, $(*)$ implies that, for any $a\in\Oo_B$
\[(b^{\dagger}b)\iota(a)(b^{\dagger}b)^{-1}=\iota(a)=\iota(a)(b^{\dagger}b)(b^{\dagger}b)^{-1}\]
Multiplying both sides by $(b^{\dagger}b)$ on the right, we have, for any $a\in\Oo_B$,
\[(b^{\dagger}b)\iota(a)=\iota(a)(b^{\dagger}b).\]
Thus $b^{\dagger}b\in B'$, and since 
$b^{\dagger}b\in \End_S\mcA\otimes\Q$, we have $b^{\dagger}b\in H(\Q)$. Since $(b^{\dagger})^{\dagger}=b$, we have $(b^{\dagger}b)^{\dagger}=b^{\dagger}b$, and thus we also have $b^{\dagger}b\in H^{\dagger}(\Q)\subset\underline{\Aut}(\mcA)^{\dagger}(\Q)$. 
Recall the map we defined earlier 
\begin{align}\label{map-varp}
    \wp: \underline{\Aut}(\mcA)&\to\underline{\Aut}(\mcA)^{\dagger}\\
    m&\mapsto mm^{\dagger}
\end{align}
and recall that we also defined $\tilde{H}:=\wp^{-1}(H^{\dagger})\subset\underline{\Aut}(\mcA)$. Thus we have
\[b\in\tilde{H}(\Q)\subset\underline{\Aut}(\mcA)(\Q).\]
On the other hand, since $\iota'=b\iota b^{-1}$ sends $\Oo_B$ to $\End(\mcA)$, we also have $b\in \underline{\Aut}(\mcA)(\Z)$. 
Moreover, we can consider $b\mod H\in (\underline{\Aut}(\mcA)/H)(\Z)\subset L(\Z)$. 
By Proposition \ref{finiteness-OKS-flat-gp}, $H^1(\Spec\Z,H)$ is a finite set. 

We then apply Lemma \ref{Lemma 5} above to the following sequence
\[\underline{\Aut}(\mcA)(\Z)\to(\underline{\Aut}(\mcA)/H)(\Z)\to H^1(\Spec\Z,H)\]
and conclude that $\Aut(\mcA)(\Z)$ has only finitely many orbits on $(\Aut(\mcA)/H)(\Z)$.

Since any such $\iota'$ is determined by its conjugating element $b\in\tilde{H}(\Q)$ which in fact comes from $b\in\tilde{H}(\Z)$ (because we fixed an $\iota$ at the beginning), 
to show that there are only finitely many possibilities for $\iota'$, it suffices to show that there are finitely many 
\[b\in\tilde{H}(\Z)=\Aut(\mcA)(\Z)\cap\tilde{H}(\Q).\]
And since $\Aut(\mcA)(\Z)$ has finitely many orbits on $(\Aut(\mcA)/H)(\Z)$, it suffices to show that $\tilde{H}(\Z)/H(\Z)$ is finite.\\ 
On the other hand, 
$I$ is compact, thus $I(\Z)$ is compact and discrete, so $I(\Z)$ must be finite. Thus to show that $\tilde{H}(\Z)/H(\Z)$ is finite, it suffices to show that $I(\Z)\backslash\tilde{H}(\Z)/H(\Z)$ is finite. 
\end{subproof}

\textit{Step II: The theorem further reduces to showing that $H^{\dagger}(\Z)/H(\Z)$ is finite. }
\begin{subproof}
Recall the map \ref{map-varp} we defined earlier.  
Now we restrict the map $\wp$ to the preimage of $H^{\dagger}$ (which we called $\tilde{H}$), and consider
\begin{align*}
    \wp: \tilde{H}&\to H^{\dagger}\\
    m&\mapsto m^{\dagger}m
\end{align*}
Since $I=\ker\wp$, the above map factors through the projection $\tilde{H}\to I\backslash \tilde{H}$, 
and thus the map $I\backslash \tilde{H}\hookrightarrow H^{\dagger}$ identifies $I\backslash \tilde{H}$ with a subspace of $H^{\dagger}$. This identification is $H$-equivariant 
when we let $H$ act on $H^{\dagger}$ via $m\mapsto h^{\dagger}mh$ for $h\in H$ and $m\in H^{\dagger}$. \\
Thus in order to show that $I(\Z)\backslash\tilde{H}(\Z)/H(\Z)$ is finite, it suffices to show that $H^{\dagger}(\Z)/H(\Z)$ is finite.
\end{subproof}

\textit{Step III: In order to show that $H^{\dagger}(\Z)/H(\Z)$ is finite, we first show the map $\vartheta$ (to be defined below) is smooth and surjective over $\Z[\frac{1}{2}]$.}
\begin{subproof}
Consider the map 
\begin{align}\label{7.1}
    \vartheta: H\times H^{\dagger}&\to H^{\dagger}\times H^{\dagger}\\
    (h,m)&\mapsto (h^{\dagger}mh,m)
\end{align}
Consider the fibre of $\vartheta$ over a point $(m,m')$ and we obtain a scheme 
\[\vartheta^{-1}(\{(m,m')\})=\{(h,m')\in H\times \{m'\}: h^{\dagger}m'h=m\}:=H_{m,m'}\times\{m'\}\]
where we denote the scheme 
$H_{m,m'}:=\{h\in H: h^{\dagger}m'h=m\}$. 
Consider the involution on $\Oo_{B'}$ given by
\begin{align*}
    \dagger(m):\Oo_{B'}&\to \Oo_{B'}\\
    h&\mapsto m^{-1}h^{\dagger}m
\end{align*}
Define the scheme 
$H^{\dagger(m)}:=\{h\in H: h^{\dagger}mh=m\}=\{h\in H: h^{\dagger (m)}h=1\}$. 
Now, the scheme $H_{m,m'}$ is either empty, or a torsor under $H^{\dagger(m)}$. 
The generic fibre of $H^{\dagger(m)}$ has a connected component of identity which is reductive. 
To compute the dimension of the fibres 
of $H^{\dagger(m)}$ over $\Z[1/2]$, 
we study its Lie algebra $\Lie (H^{\dagger(m)})$. 
Note that 
$\Lie({H^{\dagger(m)}})=(\Lie H)^{\dagger(m)=-1}$.

Now, since $\dagger(m)$ and $\dagger$ are conjugate to each other, they have the same number of eigenvalues equal to $-1$ when acting on $B'$.  
Thus the dimension of this Lie algebra $\Lie(H)^{\dagger(m)=-1}$ over any point of $\Spec\Z[1/2]$ is the same as the dimension of $(\Lie H)^{\dagger=-1}$, which is equal to $\dim H-\dim H^{\dagger}$. 
On the other hand, since
\[\dim(\Lie H)^{\dagger=-1}\geq \dim H^{\dagger(m)}\geq \dim H-\dim H^{\dagger}\]
Thus all inequalities in the above achieve equality: 
\[\dim H^{\dagger(m)}=\dim (\Lie H)^{\dagger=-1}=\dim \Lie (H^{\dagger(m)}).\]
Thus $H^{\dagger(m)}$ is smooth over $\Z[1/2]$ and $\dim H^{\dagger(m)}=\dim H-\dim H^{\dagger}$. 
Now, since $H_{m,m'}$ is a torsor under $H^{\dagger(m)}$ which is smooth of dimension independent of $m$, we have that $H_{m,m'}$ is also smooth of dimension independent of $m$ or $m'$. 
Therefore, the map $\vartheta$ has smooth equidimenisonal fibres $H_{m,m'}$ over $\Z[1/2]$, in particular $\vartheta$ is flat over $\Z[1/2]$, as it is a map of smooth $\Z$-schemes. 

To show that $\vartheta$ is surjective over $\Z[1/2]$: we first note that the above argument implies that $\vartheta$ is an open map. 
In particular, we consider for any $m\in H^{\dagger}$, the restriction
    $\vartheta|_{H^{\dagger}\times m}: H^{\dagger}\times m\to H^{\dagger}\times m$. 
Therefore, in any fibre of $H^{\dagger}$ over a point of $\Spec\Z[1/2]$, the orbits of $H$ acting on $H^{\dagger}$ are open. 
Since $H^{\dagger}$ has connected fibres, this implies that there is only one orbit, and thus $H$ acts transitively on the fibres of $H^{\dagger}$. Hence $\vartheta$ is surjective over $\Z[1/2]$. 
\end{subproof}

\textit{Step IV: The situation over $\Z$: reduces to showing that $F_m(\Z)/H(\Z)$ is finite.}
\begin{subproof}
Let $(H^{\dagger})_I$ denote the blow up of $H^{\dagger}$ in an ideal $I$ supported over the prime $2$. 
We consider the inverse limit
$(H^{\dagger})^{\sim}=\varprojlim_I (H^{\dagger})_I$. The map $(H^{\dagger})^{\sim}\to H^{\dagger}$ induces a bijection on $\Z$-points. 
Let $m\in H^{\dagger}$, and consider the orbit map $H\to H^{\dagger}$ 
obtained by taking the fibre of $\vartheta$ over $m$. 
We denote by $H_m^{\sim}$ the proper transform of $H$ with respect to $(H^{\dagger})^{\sim}\to H^{\dagger}$. That is, $H_m^{\sim}$ is the closure of the generic fibre of $H$ in $H\times (H^{\dagger})^{\sim}$. 
By \cite[Theorem 5.2.2.]{Raynaud-Gruson}, the map 
$H_m^{\sim}\to (H^{\dagger})^{\sim}$ 
is flat, hence open. 
We denote by $F_m$ its image. This is an open set containing the point corresponding to $m$. 
Let $F=\bigcup\limits_mF_m$, an open subset of $(H^{\dagger})^{\sim}$ containing $H^{\dagger}(\Z)$. Since $(H^{\dagger})^{\sim}$ is quasi-compact (it is an inverse limit of quasi-compact schemes), $F$ is a union of finitely many $F_m$. Thus it suffices to show that $F_m(\Z)/H(\Z)$ is finite.
\end{subproof}

\textit{Step V: The finiteness of $F_m(\Z)/H(\Z)$ reduces to the finiteness of $H/\overline{H^{\dagger(m)}}(\Z)$.}

\begin{subproof}
Note that $H_m^{\sim}\to F_m$ is of finite type, thus by a theorem of Skolem (\cite{Moret-Bailly-Skolem12}, if $m'\in F_m(\Z)$, there is an $h\in H_m^{\sim}(\Oo_K)=H(\Oo_K)$ with $h\cdot m=m'$ for some number field $K$. 

We denote by $\overline{H^{\dagger(m)}}$ the closure of $H^{\dagger(m)}$ in $H$, so that this is a finite type 
flat group scheme over $\Z$. Then the image of $h$ in $H_m^{\dagger}:=H/\overline{H^{\dagger(m)}}$ is in $H/\overline{H^{\dagger(m)}}(\Z):=H_m^{\dagger}(\Z)$ since $m,m'\in H^{\dagger}(\Z)$. 
Thus in order to show that $F_m(\Z)/H(\Z)$ is finite, it suffices to show that $H/\overline{H^{\dagger(m)}}(\Z)$ is finite.
\end{subproof}
Finally, to show that $H/\overline{H^{\dagger(m)}}(\Z)$ is finite, by Lemma \ref{Lemma 5}, it suffices to show that the set of torsors $H^1(\Spec\Z,\overline{H^{\dagger(m)}})$ 
is finite, and this follows from Proposition \ref{finiteness-OKS-flat-gp}. 
Thus we have finished the proof of Theorem \ref{main_PEL_thm}. 
\end{proof}

\textit{Acknowledgement.} The author would like to thank Kestutis Cesnavicius and Mark Kisin for helpful conversations.

\bibliographystyle{amsalpha}
\bibliography{bibfile}

\providecommand{\bysame}{\leavevmode\hbox to3em{\hrulefill}\thinspace}
\providecommand{\MR}{\relax\ifhmode\unskip\space\fi MR }
\providecommand{\MRhref}[2]{%
  \href{http://www.ams.org/mathscinet-getitem?mr=#1}{#2}
}
\providecommand{\href}[2]{#2}
\begin{thebibliography}{LMB00}

\bibitem[BHC62]{Borel-Harish-Chandra}
Armand Borel and Harish-Chandra, \emph{Arithmetic subgroups of algebraic
  groups}, Ann. of Math. (2) \textbf{75} (1962), 485--535. \MR{147566}

\bibitem[GR03]{Gabber-Ramero}
Ofer Gabber and Lorenzo Ramero, \emph{Almost ring theory}, Lecture Notes in
  Mathematics, vol. 1800, Springer-Verlag, Berlin, 2003. \MR{2004652}

\bibitem[LMB00]{Laumon-Moret-Bailly}
G\'{e}rard Laumon and Laurent Moret-Bailly, \emph{Champs alg\'{e}briques},
  Ergebnisse der Mathematik und ihrer Grenzgebiete. 3. Folge. A Series of
  Modern Surveys in Mathematics [Results in Mathematics and Related Areas. 3rd
  Series. A Series of Modern Surveys in Mathematics], vol.~39, Springer-Verlag,
  Berlin, 2000. \MR{1771927}

\bibitem[MB89]{Moret-Bailly-Skolem12}
Laurent Moret-Bailly, \emph{Groupes de {P}icard et probl\`emes de {S}kolem.
  {I}, {II}}, Ann. Sci. \'{E}cole Norm. Sup. (4) \textbf{22} (1989), no.~2,
  161--179, 181--194. \MR{1005158}

\bibitem[RG71]{Raynaud-Gruson}
Michel Raynaud and Laurent Gruson, \emph{Crit\`eres de platitude et de
  projectivit\'{e}. {T}echniques de ``platification'' d'un module}, Invent.
  Math. \textbf{13} (1971), 1--89. \MR{308104}

\end{thebibliography}

\end{document}